\documentclass[sn-mathphys,Numbered]{sn-jnl}%

\usepackage{graphicx}%
\usepackage{multirow}%
\usepackage{amsmath,amssymb,amsfonts}%
\usepackage{amsthm}%
\usepackage{mathrsfs}%
\usepackage[title]{appendix}%
\usepackage{xcolor}%
\usepackage{textcomp}%
\usepackage{manyfoot}%
\usepackage{booktabs}%
\usepackage{algorithm}%
\usepackage{algpseudocode}%
\usepackage{listings}%
\newtheorem{theorem}{Theorem}

\newtheorem{lemma}{Lemma}%

\newtheorem{definition}{Definition}%

\newcommand*{\R}{{\mathbb R}}

\DeclareMathOperator*{\argmin}{argmin}

\newcommand{\sm}[2]{\begin{smallmatrix}\item1\\\item2 \end{smallmatrix}}

\def\R{\mathbb{R}}

\newcommand\ddfrac[2]{\frac{\displaystyle \item1}{\displaystyle \item2}}

\def\eps{\varepsilon}
\newcommand{\mc}[1]{\mathbb{\item1}}

\newcommand{\sol}{\mathcal{X}^\ast}

\begin{document}

\title[Accelerated methods for weakly-quasi-convex optimization problems.]{Accelerated methods for weakly-quasi-convex optimization problems.}

\author[1]{\fnm{Sergey} \sur{Guminov}}\email{sergey.guminov@gmail.com}

\author[2,3,4]{\fnm{Alexander} \sur{Gasnikov}}\email{gasnikov@yandex.ru}

\author*[2,4]{\fnm{Ilya} \sur{Kuruzov}}\email{kuruzov.ia@phystech.edu}

\affil[1]{\orgname{National Research University Higher School of Economics}, \orgaddress{\city{Moscow}, \country{Russia}}}

\affil[2]{\orgname{Moscow Institute of Physics and Technology}, \orgaddress{\city{Dolgoprudny}, \country{Russia}}}

\affil[3]{\orgname{Skoltech}, \orgaddress{\city{Moscow}, \country{Russia}}}

\affil[4]{\orgname{Institute for Information Transmission Problems}, \orgaddress{\city{Moscow}, \country{Russia}}}


\maketitle

\begin{abstract}

We provide a quick overview of the class of $\alpha$-weakly-quasi-convex problems and its relationships with other problem classes. We show that the previously known Sequential Subspace Optimization method retains its optimal convergence rate when applied to minimization problems with smooth $\alpha$-weakly-quasi-convex objectives. We also show that Nemirovski's conjugate gradients method of strongly convex minimization achieves its optimal convergence rate under weaker conditions of $\alpha$-weak-quasi-convexity and quad\-ratic growth. Previously known results only capture the special case of 1-weak-quasi-convexity or give convergence rates with worse dependence on the parameter $\alpha$.
\end{abstract}


\let\oldproofname=\proofname
\renewcommand{\proofname}{\rm\bf{\oldproofname}}
\renewcommand{\leq}{\leqslant}
\renewcommand{\geq}{\geqslant}

\date{}
\maketitle

\section*{Introduction}

One of the most well-studied optimization problem classes is the class of minimization of convex functions with Lipschitz continuous derivatives. In this case, one of the most important implications of the convexity is its unimodality: the local minima of the object function are also global minima and comprise a convex subset.

However, there exist conditions weaker than convexity which still imply the unimodality of the objective function and the convergence of particular first-order minimization methods to a global minimum. One example of such condition is star-convexity, which is sufficient to derive the convergence rate of the cubic regularization of the Newton method \cite{nesterov2006cubic}. Besides, the recent works \cite{Kleinberg, zhou2018sgd} allows us to think that loss function for neural methods around minima for some modern neural networks is star convex. Another example is weak-quasi-convexity, which is a further generalization of star-convexity. 

The example of weakly-quasi-convex but not star-convex function was shown in \cite{WQC}. The authors considered recurrent neural network without non-linear activation for approximation of sequence generated by unknown model from the same class. The quadratic loss function was chosen. It was approximated by so-called idealized risk, and this new loss function is weakly-quasi-convex under natural assumption about parameters of unknown model. However, in the mentioned work the problem is solved using the stochastic gradient descent method, which is not optimal even in the convex case.

Another important example was provided in the recent work \cite{Wang_Wibisono_2023}. In this work, authors consider the generalized linear model -- combination of linear layer with non-linear activation. The paper contains proof that square loss minimization problem is quasar-convex for some typical non-linear activations like Leaky-ReLU, logistic functions and quadratic functions. 

Some first-order methods can be applied to star-convex or $\alpha$-weakly-quasi-convex problems without any modifications. The most basic examples of such methods are the gradient descent method and the stochastic gradient descent method \cite{WQC}, \cite{Gower_2021}. However, at the time of writing, no accelerated first-order methods have been shown to achieve optimal convergence  rates in the general weakly-quasi-convex setting. The purpose of this work is to construct such methods.

In this paper we obtain Theorem \ref{aSESOP} devoted to Sequential Subspace Optimization (SESOP) method. It indicates that this method has optimal convergence rate for weakly-quasi-convex functions. The main disadvantage of this method is that at each iteration we have to solve a possibly difficult auxiliary problem in $\mathbb{R}^3$. Nevertheless, this method works for any smooth weakly-quasi-convex function and does not require any additional information about function like Lipschitz constant for gradient. 

Another considered method is Nemirovski's Conjugate Gradient method. It is well-known that it has optimal convergence rate for strongly convex minimization if one use restarts technique. Theorem \ref{CG_theorem} presents a generalized result for the class of $\alpha$-weakly-quasi-convex functions satisfying the quadratic growth condition and with Lipschitz continuous gradients.

Since the first version of this paper became available online, some new results have been established.\footnote{The first version appears in 2017 as arXive paper \url{https://arxiv.org/pdf/1710.00797.pdf}. We did not publish it since we try to eliminate auxiliary subspace optimization. In 2018 motivated by our arXive paper it was done in \cite{hinder2019near}. We lost the interest for publishing. But in the last few years we fulfilled different experiments with SESOP type algorithms and have observed that these algorithms, which use subspace optimization (not line search!), converge for smooth strongly convex problems much faster than standard accelerated methods, which required as input smoothness and strong convexity constants. SESOP does not require any input information. To the best of our knowledge this observation shed a light on the classical problem -- to develop accelerated method that is adaptive in strong convexity parameter. The solution could be -- to use two or three dimensional subspace optimization at each iteration. Based on this new observation we understand that subspace optimization could be  an advantage in practice and decide to publish our arXive paper to draw attention to all these issues...} In \cite{nesterov2020} we present an accelerated first-order method with provable convergence in the weakly-quasi-convex case. However, the convergence rate established for that method has worse dependence on parameter of weakly-quasi-convexity than the methods studied in this work. In \cite{hinder2019near} a worst-case lower bound on the complexity (in terms of the required number of gradient evaluations) of a first order method is established. This bound coincides with the complexity bound established for the SESOP method in this work. The authors also present a novel first-order method with a complexity optimal up to a logarithmic factor.

The last section of this paper is devoted to numerical experiments. Note, that SESOP method does not require any function parameters and it can be considered as method that is adaptive to Lipschitz constant of gradient. Nevertheless, the adaptive to parameter of strong convexity methods are not known now. In the last section, we provide some experiments evidencing that SESOP can work faster than known optimal methods for strong convex functions. We compare its with well-known Nesterov accelerated gradient descent and considered Conjugate Gradients method with restarts. It is important, that these methods require knowledge about parameter of strong convexity. Let us highlight that we demonstrate that SESOP converges better in some special cases not only for iterations but for time too.



\section*{Preliminaries}

Throughout this paper we will be dealing with the global minimization problem \[f(x)\rightarrow\min_{x\in \R^n}.\] $f:\ \R^n \to \mathbb{R}$ is assumed to be differentiable and $L$-smooth:

\[\|\nabla f(y)-\nabla f(x)\|\leq L\|y-x\|\quad \forall x,y\in\R^n.\] We will also assume that the solution set $\mathcal{X}^\ast$ is not empty and denote $f^\ast=\min\limits_{x\in\mathbb{R}^n} f(x)$.  $\|\cdot\|$ denotes the Euclidean norm, $\langle\cdot,\cdot\rangle$ denotes the scalar product defined as $\langle x,y\rangle=\sum\limits_{i=1}^n x_i y_i$.

\section{Conditions}

In our work we will be studying a generalization of convexity called $\alpha$-weak-quasi-convexity, as defined in \cite{WQC}.

\begin{definition}\label{Def_alpha}

A differentiable function $f:\ \R^n \to \mathbb{R}$ is said to be $\alpha$-weakly-quasi-convex ($\alpha$-WQC) with respect to $x^\ast\in\mathcal{X}^\ast$ with parameter $\alpha\in(0,1]$ if for all $x\in\R^n$
\[\alpha(f(x)-f^\ast)\leqslant  \langle\nabla f(x),x-x^\ast\rangle .\]
\end{definition}

 {Note, that} $\alpha$-weak-quasi-convexity guarantees that any local minimizer of $f$ is also a global minimizer. This condition controls the distance between the graph of the function and the tangent plane to the function's graph constructed at any point. Any convex function attaining its minimum is also 1-WQC, but the converse is generally not true, even for functions of one variable. The function $f(x) = |x|(1 - e^{-|x|})$ is an example of a non-convex 1-WQC function.

To weaken strong convexity we will be using the quadratic growth (QG) condition \cite{QG,bonnans1995second}. 

\begin{definition}

A function $f:\ \R^n \to \mathbb{R}$ is said to satisfy the quadratic growth condition if for some $\mu>0$ and for all $x\in\R^n$

\[f(x)-f^\ast\geq \frac{\mu}{2}\|x-P(x)\|^2, \]where $P(x)$ is the projection of $x$ onto $\sol$.
\end{definition}

Note that the same condition appears in \cite{nesterov2006cubic} as a property of the solution set. The solution set $\sol$  is called globally non-degenerate if it satisfies the inequality in the definition of QG.

Though this condition shares some similarities with strong convexity, a non-strongly convex function may still satisfy it. Function $f(x)=(\|x\|^2-1)^2$ serves as an example. \cite{nesterov2006cubic}.

\subsection{Relationship with other conditions}

Naturally, $\alpha$-weak-quasi-convexity and the QG condition are not the only ways to weaken convexity and strong convexity. What follows is a short list of other similar conditions and their relationships with the ones used in this paper.

Let us define the Polyak--\L{}ojasiewicz condition -- another condition used to replace strong convexity in convergence arguments. Note, one of the recent nice application of this condition \cite{fatkhullin2020optimizing}.

\begin{definition}\label{PL-cond}
A function $f:\ \R^n \to \mathbb{R}$ is said to satisfy the Polyak--\L{}ojasiewicz condition if for some $\mu>0$\ and for all $x\in\R^n$

\[\frac{1}{2}\|\nabla f(x)\|^2\geq \mu(f(x)-f^\ast).\]
\end{definition}
As shown in \cite{PL}, QG is weaker than the the Polyak--\L{}ojasiewicz condition.

Following \cite{nesterov2006cubic}, we define star-convexity.

\begin{definition}
A function $f:\ \R^n \to \mathbb{R}$ is called star-convex if for any $x^\ast\in\sol$ and any $x\in\R^n$ we have
\[f(\lambda x^\ast + (1 -\lambda)x) \leq \lambda f (x^\ast) + (1-\lambda)f (x)\quad \forall x\in\R^n,\  \forall \lambda\in [0, 1].\]
\end{definition}

This definition is ideologically similar to that of $\alpha$-WQC in the way it restricts convexity to the direction towards the solution set $\sol$. In fact, for differentiable functions and with $\alpha=1$ these two definitions are equivalent.

\begin{lemma}
Let function $f:\mathbb{R}^n \rightarrow \mathbb{R}$ be differentiable. Then it is 1-WQC if and only if it is star-convexity.
\end{lemma}
\begin{proof}

$\Rightarrow$\\
Let us assume that $f(x)$ is not star-convex:
\[\exists \lambda\in(0,1),\ x\in\R^n\ s.t.\ f(\lambda x^\ast+(1-\lambda)x)-\lambda f(x^\ast)-(1-\lambda)f(x)>0. \]

By maximizing the LHS of the above inequality over $\lambda\in(0,1)$ we get that for some $\lambda^\ast\in(0,1)$ and $z=\lambda^\ast x^\ast+(1-\lambda^\ast)x$
\[\langle\nabla f(z),x-x^\ast\rangle=f(x)-f(x^\ast)\] and \[f(z)>\lambda^\ast f(x^\ast)+(1-\lambda^\ast) f(x).\] Now we note that $x-x^\ast=\frac{z-x^\ast}{1-\lambda^\ast}$ and $f(z)-f(x^\ast)>(1-\lambda^\ast)(f(x)-f(x^\ast)).$ This in turn implies that $\langle\nabla f(z),z-x^\ast\rangle<f(z)-f(x^\ast),$ so $f(x)$ is not 1-WQC.

$\Leftarrow\quad $If $f$ is star-convex, then 
\[f(x^\ast)-f(x)\geq \frac{f(x+\lambda(x^\ast-x))-f(x)}{\lambda}\quad \forall \lambda\in(0,1), x\in\R^n.\] Taking the limit $\lambda\to +0$, we obtain $f(x^\ast)-f(x)\geq \langle \nabla f(x), x^\ast-x\rangle,$ so $f(x)$ is 1-WQC.
\end{proof}

Another condition was recently introduced in \cite{csiba2017global} to generalize convexity.  Called the weak PL inequality, in our notation it may be defined as follows.

\begin{definition}
A function $f(x)$ is said to satisfy the weak PL inequality with respect to $x^\ast\in\sol$ if for some $\mu>0$ and for all $x\in\R^n$

\[\sqrt{\mu}(f(x)-f^\ast)\leq\|\nabla f(x)\|\|x-x^\ast\|.\]
\end{definition}

It immediately follows from the Cauchy-Schwarz inequality that the weak PL inequality is weaker than $\alpha$-WQC .
\section{Gradient descent method}

One of the first questions arising whenever a new condition is proposed to replace convexity is whether it's sufficient to guarantee the convergence of the gradient descent method. Fortunately, this is the case with $\alpha$-WQC.  

\begin{algorithm}
	\caption{Gradient descent}
\begin{algorithmic}[1]
    \Require {$f$, $x_0$,$T$}
    \For{$k=0$ to $T-1$}
    \State $x_{k+1}\gets x_k-\frac{1}{L}\nabla f(x_k)$
  \EndFor
 \State \Return{$x_T$}
\end{algorithmic}
\end{algorithm}

It is well known that for a convex $L$-smooth objective $f$ the gradient descent method generates a sequence $\{x_k\}$ such that 

\[f(x_k)-f^\ast = O\left(\frac{LR^2}{k}\right),\] where $R=\|x_0-x^\ast\|$. We will now provide proof of a similar result for $\alpha$-WQC objectives.

\begin{theorem}
Let the objective function $f:\R^n\to\R$ be $L$-smooth and $\alpha$-WQC with respect to $x^\ast\in\sol$. Then the sequence $\{x_k\}$ generated by the gradient descent method satisfies

\[f(x_k)-f^\ast\leq \frac{LR^2}{\alpha (k+1)},\] where $R=\|x_0-x^\ast\|$.
\end{theorem}

\begin{proof}

Any $L$-smooth function $f$ satisfies the following inequality:
\[f(y)\leqslant f(x)+\langle\nabla f(x),y-x\rangle+\frac{L}{2}\|y-x\|^2\quad \forall x,y\in R^n.\] Setting $x=x_k, y=x_{k+1},$  we obtain \begin{equation}f(x_{k+1})-f(x_k)\leq-\frac{1}{2L}\|\nabla f(x_k)\|^2.\label{grad}\end{equation} This shows that the sequence $\{f(x_k)\}_{k=0}^{\infty}$ is non-increasing. On the other hand, we have \[\frac{1}{2}\|x_{k+1}-x^\ast\|^2=\frac{1}{2}\|x_k-x^\ast\|^2-\frac{1}{L}\langle\nabla f(x_k),x_k-x^\ast\rangle+\frac{1}{2L^2}\|\nabla f(x_k)\|^2.\] Taking into account this equality and the gradient descent guarantee \eqref{grad}, we get
\[\langle\nabla f(x_k),x_k-x^\ast\rangle\leq\frac{L}{2}\|x_{k}-x^\ast\|^2-\frac{L}{2}\|x_k-x^\ast\|^2+f(x_k)-f(x_{k+1}).\]

Further, denote $\eps_k=f(x_k)-f^\ast$. The above inequality combined with the definition of $\alpha$-WQC shows that
\[\eps_k\leq \frac{1}{\alpha}\langle\nabla f(x_k),x_k-x^\ast\rangle\leq\frac{1}{\alpha}\left[\frac{L}{2}\|x_k-x^\ast\|^2-\frac{L}{2}\|x_{k+1}-x^\ast\|^2+\eps_k-\eps_{k+1}\right].\] 

Summing it up for $k=0,\ldots, T$ results in
\[\sum_{k=0}^T\eps_k\leq\frac{1}{\alpha}\left[\frac{LR^2}{2}-\frac{L}{2}\|x_{T+1}-x^\ast\|^2+\eps_0-\eps_{T+1}\right]\leq\frac{1}{\alpha}\left[\frac{LR^2}{2}+\eps_0\right].\] By $L$-smoothness of $f$, we have $\eps_0\leq \frac{LR^2}{2}$. Since the sequence $\{\eps_k\}$ is non-increasing, we have \[(T+1)\eps_T\leq\frac{LR^2}{\alpha},\] which is exactly the statement of the theorem.

\end{proof}

\section{Subspace optimization}
In 2005 Guy Narkiss et al. \cite{SESOP} presented a first-order method with optimal (up to a multiplicative constant independent of the problem) convergence rate for smooth convex problems. In this section, we will demonstrate that this method retains its convergence rate for $\alpha$-WQC $L$-smooth functions. The proof of this fact only slightly differs from the original proof in \cite{SESOP}.

Let $D_k$ ($k\geq 1$) be an $n\times 3$-matrix ($n$ is the dimensionality of the objective's domain), the columns of which are the following vectors:
$$d_k^0=\nabla f(x_k),$$
$$d_k^1=x_k-x_0,$$
$$d_k^2=\sum\limits_{i=0}^{k}\omega_i\nabla f(x_i),$$
where
\begin{align*}
\omega_i=&
\begin{cases}
1,& i=0,\\
\frac{1}{2}+\sqrt{\frac{1}{4}+\omega^2_{i-1}},& i>0.
\end{cases}
\end{align*}

These matrices will determine the subspaces over which we will minimize our objective. With $D_k$ defined this way, the algorithm takes the following form:

\begin{algorithm}[ht]
	\caption{SESOP($f$, $x_0$, $T$)}
 \begin{algorithmic}[1]     
    \Require{The objective function $f$, initial point $x_0$, number of iterations $T$}
    \For{$k=0$ to $T-1$}
    \State $\tau_k\gets\argmin\limits_{\tau\in\mathbb{R}^3} f(x_k+D_k\tau)$
  	\State $x_{k+1}\gets x_k+D_k\tau_k$
  \EndFor
 \State \Return{$x_T$}
 \end{algorithmic}
\end{algorithm}

\begin{theorem}\label{aSESOP}
Let the objective function $f$ be $L$-smooth and $\alpha$-WQC with respect to $x^\ast\in\sol$. Then the sequence $\{x_k\}_k$ generated by the SESOP method satisfies
\[f(x_{k})-f^\ast\leqslant\ \frac{2LR^2}{\alpha^2k^2}, \]
where $R=\|x^\ast-x_0\|.$

\end{theorem}
\begin{proof}

Since $\nabla f(x_k)$ belongs to the set of directions generated by $D_k$, we can use the following guarantee of gradient descent with fixed step length for $L$-smooth functions:
\begin{equation}f(x_{k+1})=\min_{s\in\mathbb{R}^3} f(x_k+D_ks)\leqslant f\left(x_k-\frac{1}{L}\nabla f(x_k)\right)\leqslant f(x_k)-\frac{\|\nabla f(x_k)\|^2}{2L}.\label{sesop-grad} \end{equation} The definition of $\alpha$-WQC may be rewritten as follows:
\begin{equation}
f(x_k)-f^\ast\leqslant\frac{1}{\alpha}\langle\nabla f(x_k),x_k-x^\ast\rangle.\label{wqc}
\end{equation}

By the construction of $x_k$, we have that $x_k$ is a minimizer of $f$ on the subspace containing the directions $x_k-x_{k-1}$ and $x_{k-1}-x_0$. It means that $\nabla f(x_k)\ \bot\ x_k-x_0$, which in turn allows us to write the following inequality instead of \eqref{wqc}:
\[f(x_k)-f^\ast\leqslant\frac{1}{\alpha}\langle\nabla f(x_k),x_0-x^\ast\rangle,\] 
Take a weighted sum over $k=0,\ldots,T-1$ for some $T\in\mathbb{N}$ with weights $\omega_k$ defined above.
\begin{equation}\sum\limits_{k=0}^{T-1} \omega_k(f(x_k)-f^\ast)\leqslant \frac{1}{\alpha}\left\langle\sum\limits_{k=0}^{T-1}\omega_k\nabla f(x_k),x_0-x^\ast\right\rangle\leqslant\frac{1}{\alpha}\left\|\sum\limits_{k=0}^{T-1}\omega_k\nabla f(x_k)\right\|R.\label{weighted-wqc}\end{equation}

Since $x_k$ is also a minimizer on the subspace containing $x_{k-1}+\sum\limits_{k=0}^{k-1}\omega_k\nabla f(x_k)$, we have that $\nabla f(x_k)\ \bot\ \sum\limits_{k=0}^{k-1}\omega_k\nabla f(x_k)$. Using \eqref{sesop-grad} and the Pythagorean theorem, we get
\[\left\|\sum\limits_{k=0}^{T-1}\omega_k\nabla f(x_k)\right\|^2=\sum\limits_{k=0}^{T-1}\omega_k^2\|\nabla f(x_k)\|^2\leqslant 2L\sum\limits_{k=0}^{T-1}\omega_k^2(f(x_k)-f(x_{k+1})). \]

Note that our choice of $\omega_k$ is equivalent to choosing the greatest $\omega_k$ satisfying
\begin{align*}
\omega_k=&
\begin{cases}
1,& k=0\\
\omega^2_k-\omega^2_{k-1},& k>0.
\end{cases}
\end{align*} Returning to \eqref{weighted-wqc} and denoting $\eps_k=f(x_k)-f^\ast$, we get
\begin{align*}
S&=\sum\limits_{k=0}^{T-1}\omega_k\eps_k\leqslant\left(\frac{2LR^2}{\alpha^2}\sum\limits_{k=0}^{T-1}\omega_k^2(\eps_k-\eps_{k+1})\right)^{-1/2}=\\&=\sqrt{\frac{2LR^2}{\alpha^2}}\sqrt{\eps_0\omega_0^2+\eps_1(\omega^2_1-\omega^2_0)+\ldots+\eps_{T-1}(\omega^2_{T-1}-\omega^2_{T-2})-\eps_{T}\omega_{T-1}^2}=\\&
=\sqrt{\frac{2LR^2}{\alpha^2}}\sqrt{\eps_0\omega_0^2+\sum\limits_{k=1}^{T-1} \epsilon_k \omega_k -\eps_{T}\omega_{T-1}^2}=\\&
=\sqrt{\frac{2LR^2}{\alpha^2}}\sqrt{S-\eps_{T}\omega_{T-1}^2}.
\end{align*}

Rewriting that, we get
\begin{equation}
    \omega_{T-1}^2\eps_{T}\leqslant S-\frac{\alpha^2S^2}{2LR^2}. \label{errorT}
\end{equation}

Note that $\omega_0=1$ and for $k\geq 1$ inequality $\omega_k=\frac{1}{2}+\sqrt{\frac{1}{4}+\omega_{k-1}^2}\geq \frac{1}{2}+\omega_{k-1}$ holds. Consequently, by induction, we obtain $\omega_k\geq\frac{k+1}{2}$. Maximizing the right-hand side of \eqref{errorT} over $S$ and using the fact that $\omega_k\geqslant \frac{k+1}{2}$, we obtain

\[\eps_{T}\leqslant\frac{2LR^2}{\alpha^2T^2}. \]

\end{proof}

\section{Nemirovski's conjugate gradients method}
Consider a quadratic minimization problem \[\phi(x)=\frac{1}{2}\langle x,Ax\rangle+\langle b,x\rangle\to \min_{x\in\R^n},\] where $A\in\R^{n\times n},\ b\in\R^n$, $\mu\|x\|\leq \|Ax\|\leq L\|x\|$. The last pair of inequalities means that $\phi(x)$ is $L$-smooth and strongly convex. Denote by $x^\ast$ the solution to this problem, by $x_0$ the initial point and by $x_k$ the points generated by the conjugate gradients method. In \cite{Nemirovski} the optimal convergence rate of the conjugate gradients method for quadratic objectives was attributed to its following five properties.
\begin{enumerate}
    \item $\phi(x_{k+1})\leq\phi(x_k)-\frac{1}{2L}\nabla\phi(x_k)$;
    \item $\nabla\phi(x_k)\perp\sum\limits_{i=0}^{k-1}\nabla\phi(x_i)$;
    \item $\nabla\phi(x_k)\perp x_k-x_0$;
    \item $\langle\nabla\phi(x_k),x^\ast-x_k\rangle\leq \phi(x^\ast)-\phi(x_k)$;
    \item $\phi(x_0)-\phi(x^\ast)\geq \frac{\mu}{2}\|x_0-x^\ast\|^2$.
\end{enumerate}

These five properties are enough to derive the convergence rate of the conjugate gradients method. A method of strongly convex optimization was then constructed to possess the same five properties.

\begin{algorithm}
	\caption{CG(f, $x_0$, $T$)}
 \begin{algorithmic}[1]
    \Require{The objective function $f$, initial point $x_0$, number of iterations $T$}
    \State $q_0\gets 0$
    \For{$k=0$ to $T-1$}
    \State $E_k\gets x_0+\text{Lin}\{x_k-x_0, q_k\}$
    \State $\hat{x}_k\gets\argmin\limits_{x\in{E_k}} f(x)$
    \State $x_{k+1}\gets\hat{x}_k-\frac{1}{L}\nabla f(\hat{x}_k)$
    \State $q_{k+1}\gets q_k+\nabla f(\hat{x}_k)$
  \EndFor
 \State \Return{} $x_T$
 \end{algorithmic}
 \end{algorithm}

The gradient descent step in this algorithm leads to the first property for the sequence $\hat{x}_k$. The next two are guaranteed by the subspace minimization step. However, the final two properties are direct consequences of the strong convexity of the objective. Now note that these 2 properties are practically the very definitions of 1-WQC and quadratic growth. This suggests that this algorithm can be generalized to the $\alpha$-WQC setting.
\begin{theorem}
\label{CG_theorem}
Let  $f$ be an $L$-smooth and $\alpha$-WQC with respect to $P(x_0)$ (the projection of $x_0$ onto $\sol$) function satisfying the quadratic growth condition with constant $\mu>0$. Then CG($f$, $x_0$, $T$) returns $x_T$ such that
\[f(x_{T})-f^\ast\leqslant\frac{3}{4}(f(x_0)-f^\ast),\] where \[T=\left\lceil\frac{4}{3\alpha}\sqrt{\frac{L}{\mu}}\ \right\rceil.\]
\end{theorem}

\begin{proof}
Denote $x^\ast=P(x_0)$. Assume $\eps_{T}>\frac{3}{4}\eps_0$, which also implies $\eps_k>\frac{3}{4}\eps_0$ for $k=1,\ldots, T$. Note, that we have the following inequality for the points $\{x\}_{j=0}^T$ generated by SESEOP:
\[f(x_0)\geq f(x_1)\geq f(\hat{x}_1) \geq f(x_2) \geq \cdots \geq f(x_{T}).\] Therefore, our assumption implies that $\eps_0\neq 0$.

The gradient descent guarantee
\[f(x_{k+1})\leqslant f(\hat{x}_k)-\frac{\|\nabla f(\hat{x}_k)\|^2}{2L} \] leads us to 
\begin{equation}
\|\nabla f(\hat{x}_k)\|^2\leqslant 2L(f(\hat{x}_k)-f(x_{k+1}))\leqslant 2L(f(x_k)-f(x_{k+1})).\label{cg-grad-norm}
\end{equation} Telescoping \eqref{cg-grad-norm} for $k=0,\ldots,T-1$, we obtain
\begin{equation}
\sum\limits_{k=0}^{T-1}\|\nabla f(\hat{x}_k)\|^2\leqslant 2L(\eps_0-\eps_{T+1})\leqslant\frac{L}{2}\eps_0.\label{cg-grad-norm-sum}
\end{equation}

By the definition of $\hat{x}_k$, $\nabla f(\hat{x}_k)\perp \hat{x}_k-x_0$. This allows us to use $\alpha$-WQC in the following way:
\begin{equation}
\langle\nabla f(\hat{x}_k),x^\ast-x_0\rangle=\langle\nabla f(\hat{x}_k),x^\ast-\hat{x}_k\rangle\leqslant \alpha (f^\ast-f(\hat{x}_k))\leqslant -\frac{3\alpha}{4}\eps_0.\label{cg-wqc-error}
\end{equation} 

Now telescoping \eqref{cg-wqc-error} for $k=0,\ldots, T-1$ and using the Cauchy-Schwarz inequality, one gets
\[-\|q_{T}\| \| x^*-x_0\|\leqslant\langle q_{T}, x^*-x_0\rangle< -\frac{3T\alpha}{4}\eps_0.\]
This inequality will allow us to obtain an upper bound on $T$, which contradicts the theorem's statement. All that remains is to get upper bounds on $\|q_{T}\|$ and $x^\ast-x_0$.

Again, by definition of $\hat{x}_k$, $\nabla f(\hat{x}_{k})\perp q_k$. By the Pythagorean theorem and \eqref{cg-wqc-error},
\[\|q_{T}\|=\left(\sum\limits_{k=0}^{T-1}\|\nabla f(\hat{x}_k)\|^2\right)^{\frac{1}{2}}\leqslant\sqrt{\frac{L}{2}\eps_0}.\]

Quadratic growth, on the other hand, implies the following upper bound
\[\|x_0-x^*\|\leqslant \sqrt{\frac{2}{\mu}\eps_0}.\]

Finally,
\[-\sqrt{\frac{2}{\mu}\eps_0}\sqrt{\frac{L}{2}\eps_0}<-\frac{3T\alpha}{4}\eps_0,\] or
\[T<\frac{4}{3\alpha}\sqrt{\frac{L}{\mu}}.\] This contradicts our choice of $T=\left\lceil\frac{4}{3\alpha}\sqrt{\frac{L}{\mu}}\ \right\rceil$.
\end{proof}

This result shows that if $f$ were $\alpha$-WQC with respect to $P(x)$ $\forall x\in \mathbb{R}^n$, we would be able to apply a restarting technique to this method. To be more precise, under such circumstances it is possible to achieve an accuracy of $\eps$ by performing $\log_{\frac{4}{3}}\eps$ cycles of $\left\lceil\frac{4}{3\alpha}\sqrt{\frac{L}{\mu}}\ \right\rceil$ iterations and using the output of each cycle as input for the next one. This means that by using Nemirovski's conjugate gradients method we may get a point $y$ such that
$f(y)-f^\ast\leq\eps$ in $O\left(\frac{1}{\alpha}\sqrt{\frac{L}{\mu}}\log{\frac{1}{\eps}}\right)$ iterations.

 {
Note, that many 
problems can be reduced to the system on nonlinear equations
$g(x) = 0.$
This system can be reduced to typically non-convex optimization problem
$\min_{x\in\R^n}\left[f(x):=\frac{1}{2}\|g(x)\|^2 \right].$
If 
$\lambda_{\min}\left(\frac{\partial g}{\partial x} \left(\frac{\partial g}{\partial x}\right)^T \right)\ge \mu,$
i.e. (see Definition~\ref{PL-cond})
$$f(x)-f^* = f(x) \le\frac{1}{2\mu}\|\nabla f(x)\|^2,$$
it's well known \cite{nesterov2006cubic,PL,gasnikov2017universal} that under additional smoothness assumptions standard non-accelerated iterative methods (Gradient Descent, Cubic Regularized Newton method etc.) converge as if $f$ to be $\mu$-strongly convex function. For accelerated methods such results are not known. So we was motivated to find such additional sufficient conditions that guarantee convergence for properly chosen accelerated methods. In this section we observe that such a condition could be $\alpha$-weakly-quasi-convexity of $f$ (see Definition  \ref{Def_alpha}).
}

\section{Numerical Experiments}

In this section, we present the results of numerical experiments for the SESOP and the CG methods. In this paper we've proved that these methods converge for weakly-quasi convex functions. But their practical efficiency could be demonstrated also for (strongly) convex problems. In different numerical experiments, we observe that these methods converges significantly faster than accelerated algorithms without line-search. So one of the reason is related with spectral properties of Hessian $\nabla^2 f(x^*)$: different accelerated line-search methods (CG-type methods) have this property \cite{gasnikov2017universal}. But we've observed that the SESOP algorithm (due to subspace minimization) does not require restarts (like CG-type methods) for strongly convex problems. To the best of our knowledge, the SESOP was not applied earlier for strongly convex problems. We've observed that the SESOP is the first such method that 1) is fully adaptive (doesn't require any input information); 2) allows CG-type acceleration related with  $\nabla f(x^*)$ properties; 3) works for both convex problems and strongly convex problems (without restarts). 


It is known that non-accelerated methods (such as the Steepest descent) converge for strong convex methods with speed $f(x_k)-f^*\leq C q^k$ for $q = \frac{\kappa-1}{\kappa+1}$ depending on conditional number $\kappa = L/\mu$ and don't require any input information (such as Lipschitz constant $L$ or constant of strong convexity $\mu$). At the same time, accelerated methods (such as Nesterov method \cite{d2021acceleration}) converge with speed $f(x_k)-f^*\leq C q_{acc}^k$ for $q_{acc}=\frac{\sqrt{\kappa}-1}{\sqrt{\kappa}+1}$, but they require at least $\mu$ as an input (or instead of $\mu$ another additional information, like $f^*$ \cite{gasnikov2017universal,barre2020complexity}).    
So, our purpose is to compare the SESOP method with the accelerated methods that require the parameter $\mu$ as an input.

Note, that for quadratic problems standard  the Conjugate gradient (CG) method \cite{d2021acceleration} doesn't require $\mu$ as an input. And so for quadratic problem we don't have advantages of using the SESOP instead of the CG. That is why for our experiments we've chosen the following non-quadratic function popular in many applications \cite{anderson1992discrete,peyre2019computational} (soft-max):
\begin{equation}
    \label{lse}
    f(x)=\log\left(\sum_{i=1}^m\exp(\langle a_i, x\rangle)\right)+\frac{\mu}{2}\|x\|^2.
\end{equation}
It is a $\mu$-strongly convex function with $L=\left(\lambda_{\max}\left(AA^\top \right)+\mu\right)$-Lipschitz gradient constant. 
Moreover, we have conditional number $\kappa = \frac{L}{\mu}=\frac{\lambda_{\max}\left(AA^\top \right)}{\mu}+1$ that approaches to 1 when $\mu\rightarrow \infty$. We will vary parameter $\mu$ in our experiments.

In all our experiments the dimension $n=300$,  parameter $m=200$. Components of the matrix $A$ are generated one time independently of a standard normal distribution. The subproblems for the SESOP and the CG on each iteration were being solved through the CVXPY package \cite{diamond2016cvxpy} with accuracy $10^{-7}$.

\begin{figure}[ht]  
\centering
\includegraphics[width=1\linewidth]{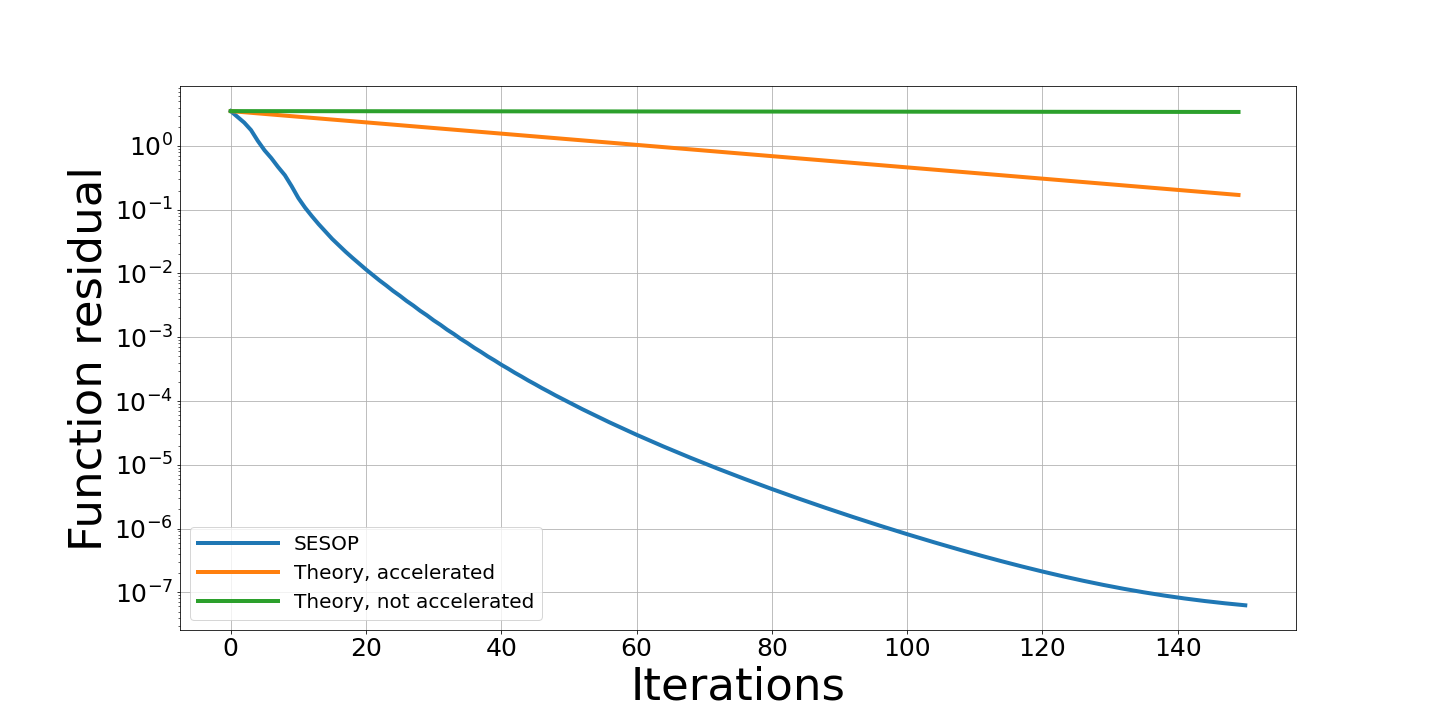} 
\caption{Comparison of convergence speed for the SESOP method and theoretical estimates for accelerated and non-accelerated methods. The parameter $l$ equals $0.1$.} \label{fig:grad_exp_lse}
\end{figure}

Firstly, we want to compare the practice speed of the SESOP method with estimates $Cq^k$ and $Cq_{acc}^k$. On the Figure  \ref{fig:grad_exp_lse} we show this comparison. The green line is non-accelerated theoretical speed $(f(x_0)-f^*)q^k$, the orange is accelerated $(f(x_0)-f^*)q_{acc}^k$. We can see that the SESOP method is significantly better than the theoretical speed for accelerated methods. Let's compare these speeds for different parameters $\mu$ and different conditional numbers correspondingly.

\begin{figure}[ht]  
\centering
\includegraphics[width=1\linewidth]{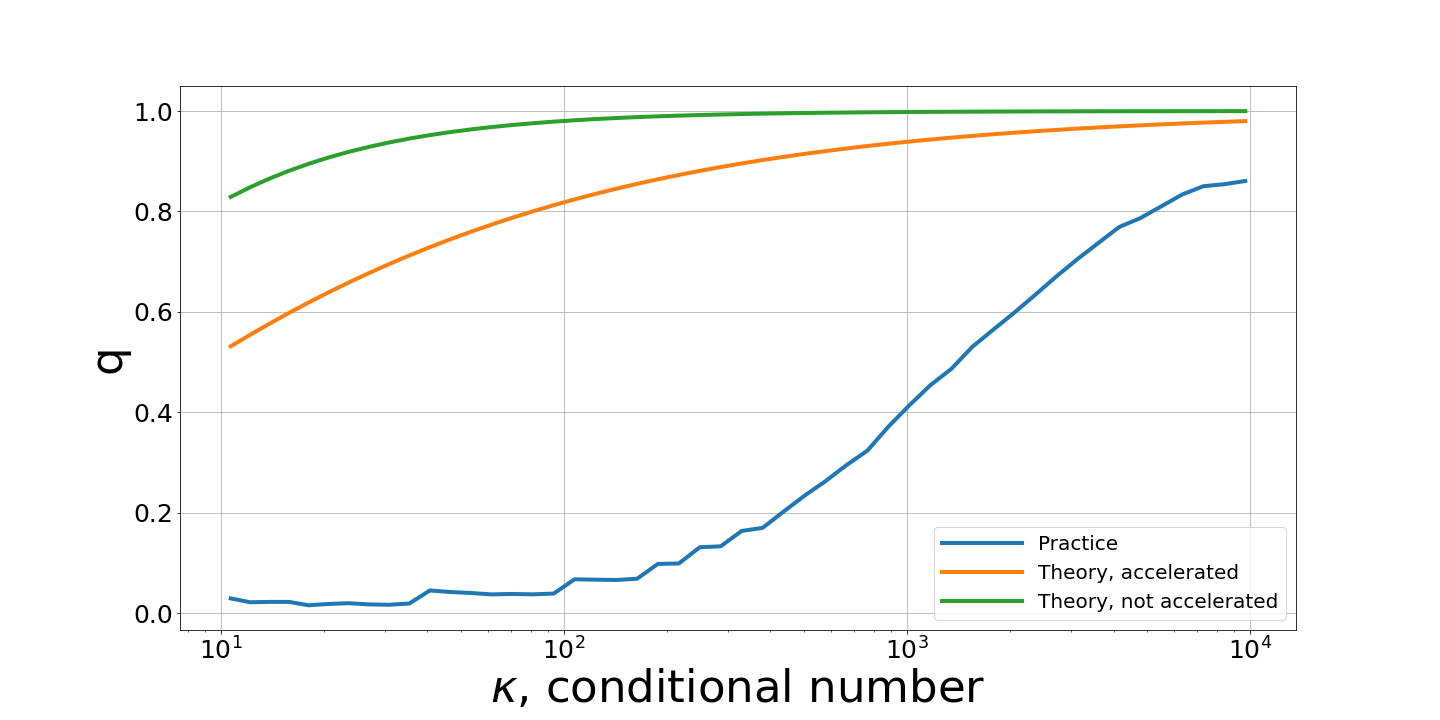} 
\caption{Comparison of $q$, $q_{acc}$ and $q_{practice}$ for the SESOP method.} \label{fig:grad_exp_qp}
\end{figure}

We will stop the SESOP method after approaching function value accuracy $10^{-8}$ and will compare parameter  $q_{practice}=\left(\frac{f_k-f^*}{f_0-f^*}\right)^{\frac{1}{k}}$, where $f_k = f(x_k)$, with $q$ and $q_{acc}$ for different conditional number. The results are presented in Figure \ref{fig:grad_exp_qp}. We can see that $q_{practice}$ is significantly less than $q_{acc}$ for all conditional numbers. It means that the SESOP method converges for this function significantly faster than the theoretical speed for accelerated methods for all conditional numbers. Moreover, the significant gain we can see for the small conditional numbers that are less than $10^2$.

Finally, we want to compare the SESOP method with proposed in this article the CG method (with and without restarts) and Nesterov method in the following form \cite{bubeck2015convex}:
$$x_{k+1}=y_k-\frac{1}{L}\nabla f(y_k),$$
$$y_{k+1}=x_{k+1}+\frac{\sqrt{L}-\sqrt{\mu}}{\sqrt{L}+\sqrt{\mu}}(x_{k+1}-x_k).$$

\begin{figure}[ht]  
\centering
\includegraphics[width=1\linewidth]{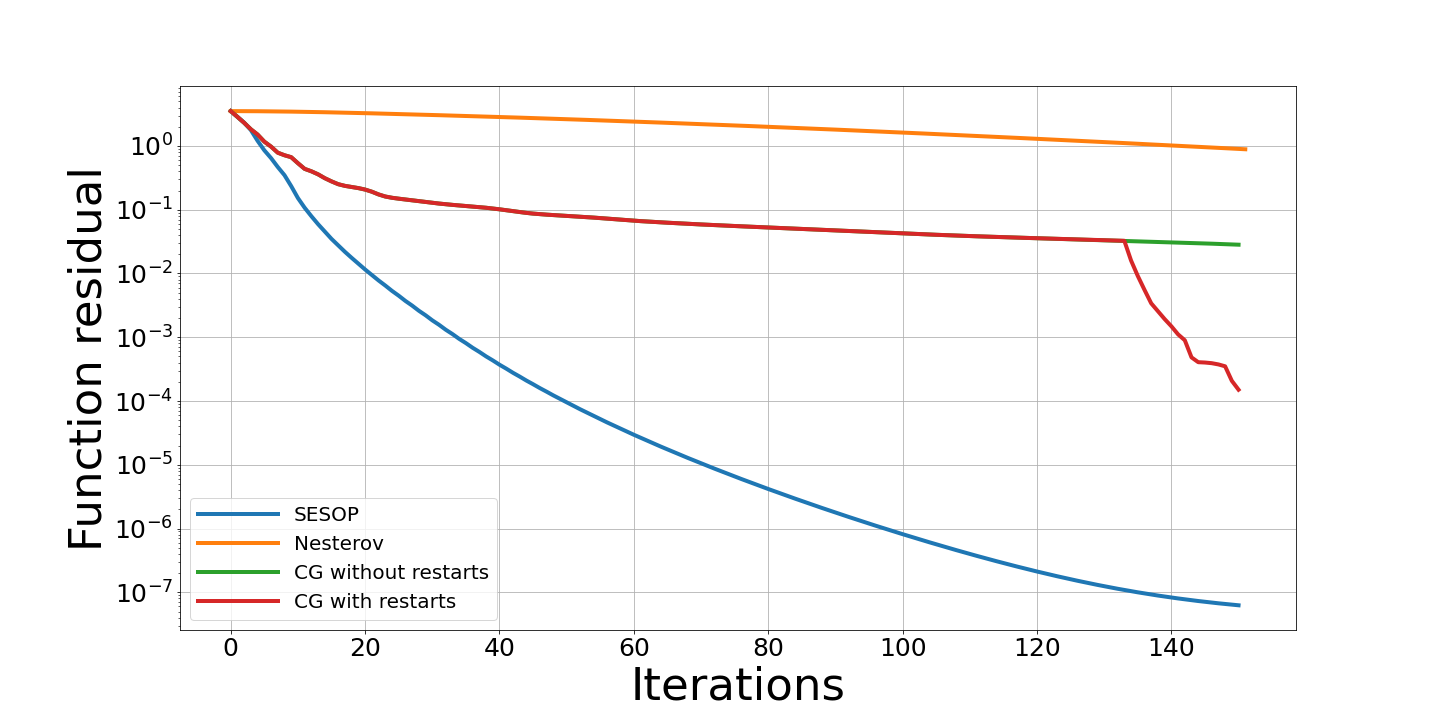}  
\caption{Comparison of the SESOP method, CG method with optimal restart from Theorem \ref{CG_theorem} and without restarts and Nesterov method. The conditional number  $\kappa\sim 10^4$} 
\label{fig:grad_exp_lse_nest1} 
\end{figure}

The results of this comparison for different conditional number is presented in Figure \ref{fig:grad_exp_lse_nest1} and \ref{fig:grad_exp_lse_nest2}. We can see that the SESOP method is better than other methods in the both cases.

\begin{figure}[ht]  
\centering
\includegraphics[width=1\linewidth]{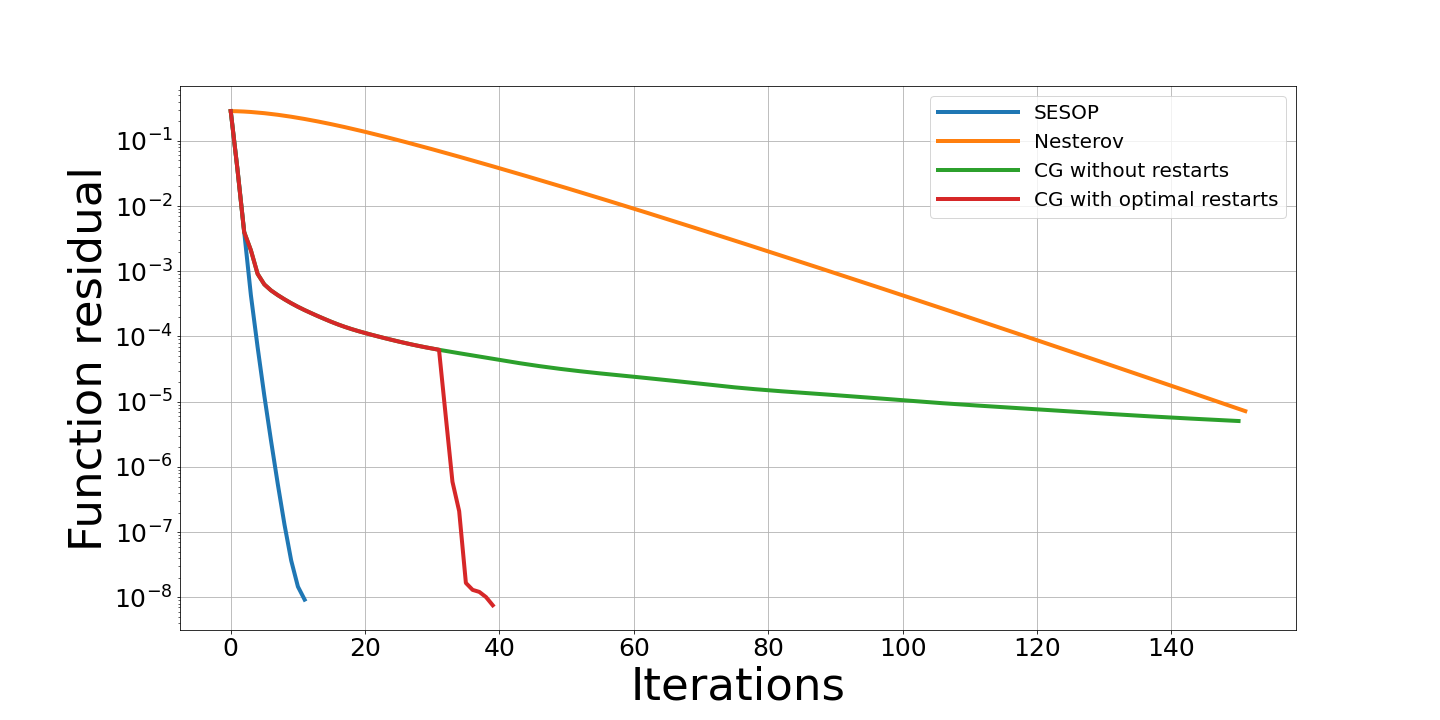}  
\caption{Comparison of the SESOP method, CG method with optimal restart from Theorem \ref{CG_theorem} and without restarts and Nesterov method. The conditional number $\kappa\sim 400$} 
\label{fig:grad_exp_lse_nest2} 
\end{figure}

Moreover, in Figure \ref{fig:grad_exp_lse_nest1} we can see that the SESOP method converges significantly better than other accelerated methods. So, it approaches accuracy $10^{-7}$ after 150 iterations when the nearest method (CG with restarts) approaches only $10^{-4}$. Also, we can see that the SESOP method and CG converges significantly faster for high conditional number (see Figure \ref{fig:grad_exp_lse_nest2}).

Besides, the CG method without restarts works significantly worse than CG with restarts in both cases.

Further, we compare time of the considered methods. Again, we take the problem of minimization \eqref{lse} and compare two methods. The first is Nesterov method with proven linear convergence rate for strong convex function. The second is researched in this article SESOP method. We demonstrate that it can achieve convergence rate in sense of iteration in practice as accelerated method. Nevertheless, it requires solving a low-dimensional problem on each iteration. Further, we demonstrate that SESOP method can demonstrate the same or better time than accelerated methods for some problems with specific structure .

We will use the Ellipsoid method for internal problem. Note, that this method can achieve accuracy $\varepsilon$ through $O\left(\log \frac{1}{\varepsilon}\right)$ operations. Besides, in the work \cite{Stonyakin_SESOP} it was shown that SESOP method is robust to inaccurate solution of internal problem on each iteration.

Ellipsoid method needs gradient at a point $\tau\in \mathbb{R}^3$ on each iteration. The gradient is 3-dimensional vector, and it can be calculated efficiently in some special cases. In particular, in the problem  the problem \eqref{lse} we have $\nabla_{\tau} f(x+D\tau)=(AD)^\top \text{softmax}(Ax+AD \tau) + \mu (D^\top x + D^\top D \tau), $ where $\text{softmax}(x)=\frac{\exp(x_i)}{\sum\limits_{j=1}^n \exp(x_j)}$. If vectors $Ax, D^\top x$ and matrices $D^\top D, AD$ then the complexity of one gradient with respect to $\tau$ is $O(m).$ In this case the complexity of each iteration in SESOP method is $O\left(mn+m\log \frac{1}{\varepsilon}\right)$. In the case, when required accuracy is not high ($\log \frac{1}{\varepsilon}\ll n$) it is near to the complexity of one gradient calculation with respect to $x$ - $O(mn)$.

\begin{table}[h]
    \centering
    \begin{tabular}{|c|c|c|}
    \hline
         m&  Nesterov & SESOP\\
    \hline
         10 & 3.1 & 0.7\\
         100 & 3.8& 3.1\\
         1000 & 18.7& 19.4\\
    \hline
    \end{tabular}
    \caption{Comparison of time (s) for Nesterov method and SESOP with Ellipsoid method.}
    \label{tab:results}
\end{table}
In the table \ref{tab:results}, we can see comparison  \eqref{lse} for different $m$ of Nesterov's method and SESOP with Ellipsoid method for internal problems. We run this method for several values $m$. The stopping criterion for all methods is $\|\nabla f(x)\|\leq 10^{-7}$.  We can see that SESOP method has better performance when gradient with respect to $\tau$ can be computed significantly faster than gradient with respect to $x$ in initial space. Moreover, the considered method does not require additional information about parameters of strong convexity or smoothness, unlike Nesterov method.

\section{Discussion}

Even though the SESOP and the CG methods presented above are optimal in terms of the amount of iterations required to achieve the desired accuracy, each iteration involves solving a subproblem over $\mathbb{R}^2$ or $\mathbb{R}^3$. However, since all the conditions replacing convexity and strong convexity in our paper involved some global minimizer $x^\ast$, which may not belong to the domain of any of these subproblems, they may be considered to be general non-convex optimization problems. Not only are such problems much more difficult than convex ones, the above convergence analyses relied on these subproblems to be solved exactly. This means that these theoretically optimal procedures can not be efficiently implemented directly. 

 {The first draft of this paper motivated two different collectives to improve 
the SESOP, that required to solve difficult subproblems. In the paper \cite{bu2020note} authors at each iteration have to solve a subproblem over $\mathbb{R}^1$. In the paper \cite{hinder2019near} authors at each iteration have an exact solution of a subproblem. Moreover, in \cite{hinder2019near} it was shown that our estimate from Theorem~\ref{aSESOP} is tight (i.e. corresponds to the lower bound from \cite{hinder2019near}). So these papers (joint with our paper) developed optimal (up to a numerical constant) methods in terms of gradient oracle calls for the class of $\alpha$-WQC target functions $f$, see Definition \ref{Def_alpha}.}

 {Note, that the considered in this paper classes of non-convex functions that can be optimized almost as convex ones. The searching of such classes of functions is one of the most popular direction in modern non-convex optimization, see recent survey \cite{danilova2020recent} and references there in. Here we just mention additional several classes: all local minima are good \cite{ge2016matrix}, class of function that can be considered as noisy convex quadratic functions \cite{bazarova2020linearly}.  Note also, that the class of unimodal functions (functions that has unique extremum that is global minima) in terms of convergence in function value is far from the class of convex function \cite{gasnikov2017universal}.}

Finally, let us discuss application of such methods for distributed optimization. The distributed optimization problem $\min\limits_{x} \sum\limits_{i=1}^m f_i(x)$ where functions $f_i$ are written on $m$ different nodes, can be reformulated as $\min\limits_{X: WX=0} \sum\limits_{i=1}^m f_i(x_i).$ For this problem we can construct the dual problem $\min\limits_{Y} \left[F^*(WY):=\sum\limits_{i=1}^m f_i^*\left([WY]_i\right)\right].$ We can apply SESOP method for this problem. In this case, each iteration of SESOP will have the following form $\tau_{k+1} = \argmin F^*(WY_k + WD_k \tau)$. Note, that after one communication over variable $Y_k$ and matrix of directions $D_k$ the method requires solving only a distributed low-dimensional problem. It can allow to significantly decrease communication complexity when initial problem has high dimension.

\subsection*{Acknowledgment}

The authors would like to thank Arkadi Nemirovski for some important remarks.


This work was supported by a grant for research centers in the field of artificial intelligence, provided by the Analytical Center for the Government of the Russian Federation in accordance with the subsidy agreement (agreement identifier 000000D730321P5Q0002) and the agreement with the Moscow Institute of Physics and Technology dated November 1, 2021 No. 70-2021-00138.

\bibliography{sn-bibliography}

\end{document}